\documentclass[12pt,a4paper]{amsart}
\usepackage{graphicx}
\usepackage{color}
\usepackage{enumerate}
\usepackage{amssymb}
\usepackage[colorlinks=true]{hyperref}
\hypersetup{urlcolor=blue, citecolor=blue, draft=false}

\newtheorem{theorem}{Theorem}[section]
\newtheorem{lemma}{Lemma}[section]
\newtheorem{corollary}{Corollary}[section]

\theoremstyle{remark}

\theoremstyle{remark}\newtheorem{remark}{Remark}[section]

\usepackage{subfigure}
\theoremstyle{definition}

\begin{document}
\title[Radii for sections of functions convex in one
direction]{Radii for sections of functions convex in one direction}

\author[Prachi  Prajna Dash, Jugal  Kishore  Prajapat and Naveen Kumari]{Prachi Prajna Dash$^\ddagger$, Jugal Kishore Prajapat$^\dagger$ and Naveen Kumari}

\address{$^\ddagger$Department of Mathematics, Central University of Rajasthan, Bandarsindri, Kishangarh-305817, Dist.- Ajmer, Rajasthan, India}
\email{prachiprajnadash@gmail.com, jkprajapat@gmail.com}

\date{}

\begin{abstract}
Let $\mathcal{G}(\alpha)$ denote the family of functions $ f(z)$ in the open unit disk $\mathbb D :=\{z\in\mathbb{C}: |z|<1\}$ that satisfy $ f(0)=0= f'(0)=1$ and
\[\Re \left(1+ \dfrac{z f''(z)}{ f'(z)}\right)<1+\dfrac{\alpha}{2} , \quad z\in \mathbb D.\]
We determine the disks $|z|<\rho_n$ in which sections $ s_n(z; f)$ of $ f(z)$ are convex, starlike, and close-to-convex of order $\beta\;(0\le \beta< 1)$. Further, we obtain certain inequalities of sections in the considered class of functions.
\end{abstract}


\maketitle
\section{Introduction}
\setcounter{equation}{0}

Let $\mathcal{A}$ denote the family of analytic functions in the open unit disk $\mathbb{D}=\{z\in\mathbb{C}: |z|<1\}$, satisfying $ f(0)=0= f'(0)-1$, and having the form
\begin{equation}\label{taylor}
 f(z)=z+\sum_{k=2}^\infty a_k \,z^k,  \quad a_k \in \mathbb{C}.
\end{equation}
The subclass of $\mathcal{A}$ consisting of all univalent functions $ f(z)$ in $\mathbb{D}$ is denoted by $\mathcal{S}$. A domain $D$ in $\mathbb{C}$ is called starlike with respect to origin or simply starlike if every line segment joining the origin to any point in $D$ lies completely inside $D$. A function $ f\in \mathcal{A}$ is called starlike if $ f(\mathbb{D})$ is a starlike domain. The subclass of starlike functions in $\mathcal{A}$ is denoted by $\mathcal{S}^*$ and satisfies the differential inequality $\Re \left( z  f'(z)/ f(z) \right)>0$, for all $z\in\mathbb{D}$. A function $ f\in \mathcal{A}$ is convex if $ f(\mathbb{D})$ is starlike with respect to every point inside the domain $ f(\mathbb{D})$. The subclass of convex functions in $\mathcal{A}$ is denoted by $\mathcal{K}$ and satisfies the inequality $\Re \left(1+z  f''(z)/ f'(z)\right)>0$, for all $z\in\mathbb{D}$. For a given $0\le \beta<1$, a function $ f \in \mathcal{A}$ is called a starlike function of order $\beta$ and denoted by $ f \in \mathcal{S}^*(\beta)$ if and only if $\Re \left(z  f'(z)/ f(z) \right)>\beta$ for all $z\in \mathbb{D}$. Similarly, for a given $0\le \beta<1$, a function $ f \in \mathcal{A}$ is called a convex function of order $\beta$, denoted by $ f \in \mathcal{K}(\beta)$ if and only if $\Re \left(1+z  f''(z)/ f'(z)\right)>\beta$ for all $z \in \mathbb{D}$. A domain $D$ is called convex in the direction $\theta \; (0\leq \theta<\pi)$ if every line parallel to the line through $0$ and $e^{i\theta}$ has a connected or empty intersection with $D$. A function $ f\in \mathcal{A}$ is known to be convex in the direction $\theta \; (0\leq \theta<\pi)$ if it maps $\mathbb{D}$ to a domain convex in the direction $\theta$.

A function $ f\in \mathcal{A}$ is close-to-convex in $\mathbb{D}$ if $\mathbb{C} \setminus  f(\mathbb{D})$ can be represented as a union of non-crossing half lines. For every convex function $\psi$, a close-to-convex function $ f\in \mathcal{A}$ satisfies the inequality $\Re \left( e^{i\theta} f'(z)/\psi'(z)\right)>0$ for all $\theta \in \mathbb{R}$. The subset of $\mathcal{A}$ consisting of close-to-convex functions is denoted by $\mathcal{C}$. It is well known that close-to-convex functions are univalent in $\mathbb{D}$, but not necessarily the converse. In particular, if $ f \in \mathcal{A}$ and $\Re(e^{i\theta} f'(z))>0$ for $z\in \mathbb{D}$ and $\theta \in \mathbb{R}$, then $ f$ is close-to-convex and hence univalent in $\mathbb{D}$. It is easy to note that $\mathcal{K} \subset \mathcal{S}^* \subset \mathcal{C} \subset \mathcal{S}$ \cite{dur}. Now again, for a given $0\le \beta<1$, a function $ f \in \mathcal{A}$ is called close-to-convex of order $\beta$, denoted by $ f\in \mathcal{C}(\beta)$, if $\Re \left(e^{i\theta} f'(z)/\psi'(z)\right)>\beta$ for all $\theta \in \mathbb{R}$. 

A function $ f$ is subordinate to the function $\psi$, written as $ f \prec \psi,$ if there exists a Schwarz function $w$, with $w(0)=0$ and $|w(z)| \leq |z|$, such that $ f(z)=\psi(w(z))$ for all $z \in \mathbb{D}$. If $\psi \in \mathcal{S}$, then the following relation holds true:
\[ f(z) \prec \psi(z) \quad \Longleftrightarrow \quad  f(0) = \psi(0) \;\;\mbox{and} \;\;  f(\mathbb{D}) \subset \psi(\mathbb{D}).\] 

Let $E\subset \mathcal{A}$ and $P$ denote the property of the image of $ f \in E$. The largest $\gamma_ f$ such that $ f \in E$ has a property $P$ in each disk $\mathbb{D}_{\rho}= \{z \in \mathbb{C}: \, |z|<\rho\}$ for $0<\rho\leq \gamma_ f$ is known as the radius of the property $P$ of $ f$. The number $\gamma := \mbox{inf}\{\gamma_ f: \;  f \in E\}$ is the radius of the property $P$ of the class $E$. If $\mathcal{G}$ is the class of all $ f \in E$ characterized by the property $P$, then $\gamma$ is called the $\mathcal{G}$-radius of the class $E$.

The study of Umezawa \cite{um} showed that a function $ f$ is locally univalent, having the form as in \eqref{taylor} and satisfying
\begin{equation}\label{uiq}
\gamma > \Re \left( 1+ \dfrac{z  f''(z)}{ f'(z)} \right) > -\;\dfrac{\gamma}{2\gamma-3},
\end{equation}
where $\gamma$ is a real number not less than $3/2$, then $ f$ is analytic and univalent in $\mathbb{D}$. Moreover, $ f$ maps every subdisk of $\mathbb{D}$ into a domain convex in one direction. Restricting $\gamma$ to $3/2$ (as many cases of \eqref{uiq} can be drawn by allowing different values of $\gamma\geq 3/2$), the subclass $\mathcal{G}$ of $\mathcal{S}$ satisfying
\begin{equation}\label{geq}
\Re \left( 1+ \dfrac{z  f''(z)}{ f'(z)} \right) < \dfrac{3}{2}, \quad z\in \mathbb{D},
\end{equation}
is obtained. Ozaki \cite{oz} showed that functions in $\mathcal{G}$ are univalent in $\mathbb{D}$. Thereafter, Singh and Singh {\cite[Theorem 6]{sig}} showed that functions in $\mathcal{G}$ are starlike in $\mathbb{D}$. 

The $n$-th section (partial sum) of $ f\in \mathcal{A}$ is the polynomial 
\begin{equation}\label{partial_sum}
 s_n(z; f)=z+\sum_{k=2}^n a_k \,z^k, \quad z\in \mathbb{D}.
\end{equation} 
The second section $ s_2(z; f)$ of the Koebe function $k(z)=z/(1-z)^2$ is univalent in $\mathbb{D}_{1/4}$. But as $ s'_2(-1/4; f)=0$, it is not univalent in any disk larger than $\mathbb{D}_{1/4}$. Proceeding from it, Szeg\"{o} \cite{szego} proved that $ s_n(z; f)$ of $ f \in\mathcal{S}$ are univalent in the disk $\mathbb{D}_{1/4}$, and the number $1/4$ cannot be replaced by a larger one. Somewhat later, Robertson \cite{robertson} showed that $ s_n(z; f)$ is univalent in $|z|<1-(3/n)\log n$ for all $n \ge 5$. The Koebe function is an extremal for many subclasses of $ f\in\mathcal{S}$, but Bshouty and Hergartner \cite{bshouty} showed that the Koebe function is not extremal for the radius of univalency of $ s_n(z; f)$ for $ f \in \mathcal{S}$. As shown in \cite[p. 246]{dur}, the exact (largest) radius of univalency of $ s_n(z; f)$ for $ f \in \mathcal{S}$ remains an open problem. Therefore, this is an area of interest.

Recently, Obradovi\'{c} and Ponnusamy \cite{obra2} showed that each section $ s_n(z; f)$ of $ f \in \mathcal{G}$ is starlike in the disk $|z|\leq 1/2$ for $n\geq 12$ and $\Re( s_n'(z))>0$ holds in the disk $|z|\leq 1/2$ for $n\geq 13$. Results for the best approximation of the generalized version of class $\mathcal{G}$ have been obtained in \cite{kar}. Also, it is shown in \cite{mah} that the radius of convexity of class $\mathcal{G}$ is $\mathbb{D}_{1/2}$ and it is irreplaceable with a larger disk.

The generalized version of the class $\mathcal{G}$ can be obtained from \cite{obra}, which is denoted as $\mathcal{G}(\alpha)$, consisting of locally univalent functions $ f \in \mathcal{A}$, which satisfies
\begin{equation}\label{geqa}
\Re \left( 1+ \dfrac{z  f''(z)}{ f'(z)} \right) < 1+\dfrac{\alpha}{2},\;\; z\in \mathbb{D},
\end{equation}
for $\alpha\in \mathbb{R}$. Note that $\mathcal{G}(\alpha)\ne \emptyset$ for $\alpha>0$, as the identity function satisfies \eqref{geqa} for $\alpha>0$. Here the class $\mathcal{G}\equiv \mathcal{G}(1)$. Many results, including sharp coefficient bounds for $ f\in \mathcal{G}(\alpha)$, sharp bounds for the Fekete-Szeg\"{o} functional for functions in $\mathcal{G}(\alpha)$ with complex parameters, a convolution characterization for functions $ f\in \mathcal{G}(\alpha)$, and sufficient coefficient conditions for $ f$ to be in $ \mathcal{G}(\alpha)$, have been shown in \cite{obra}. Also, as shown in \cite{obra}, each section $ s_n(z; f)$ of $ f\in \mathcal{G}$ is starlike in the disk $|z|\le 1/2$ for $n \ge 11$ and close-to-convex (and hence univalent) in the disk $|z|\le 1/2$ for $n \ge 11$. 

As discussed recently in \cite{kar}, $\mathcal{G}(\alpha)\subset \mathcal{G}$ whenever $\alpha\in(0,1]$. That is, $ f\in \mathcal{G}(\alpha)$ is starlike and univalent in $\mathbb{D}$, whenever $\alpha\in(0,1]$. Kargar et~al. \cite{kar} showed (with the help of an example) that $ f\in \mathcal{G}(\alpha)$ doesn't contain only univalent functions in $\mathbb{D}$, when $\alpha\ge 2$. They \cite{kar} showed that the radius of univalence of $\mathcal{G}(\alpha)$ is at least $1/\alpha$ when $\alpha\in [1,4.952)$. Also, they have established the best approximation of an analytic function by functions in a subclass of $\mathcal{G}(\alpha)$.

Here we obtain radii of convexity, starlikeness, and close-to-convexity of order $\beta$ for all sections ($ s_n(z; f)$, $n\ge 2$) of functions in the class $\mathcal{G}(\alpha)$ along with some properties of it. \textit{Section 2} contains all known results that were used in our obtained results. \textit{Section 3} is dedicated to all of our obtained results with their detailed proofs.

 \medskip
\section{Known results}

We use the following known results in our obtained results: 
\begin{lemma}\cite{obra}\label{lemma2.1}
If $ f(z)=z+\sum_{k=2}^{\infty} a_k z^k \in \mathcal{G}(\alpha)$, then
\[|a_n|\leq \dfrac{\alpha}{n(n-1)}, \quad n \ge 2.\]
The inequality is sharp and equality is attained for a function $ f$, where \linebreak
$ f'(z)=(1-z^{n-1})^{\alpha/(n-1)},$ $n \ge 2$.
\end{lemma}

\begin{lemma}\cite{obra, kar} 
If $ f(z) \in \mathcal{G}(\alpha)$, then the following inequalities and subordination results hold:
\begin{align}
&\left| \dfrac{z  f''(z)}{\alpha  f'(z)}+ \dfrac{|z|^2}{1-|z|^2} \right| \le \;  \dfrac{|z|}{1-|z|^2}, \label{lemma2.2d} \\
& f'(z)  \prec   \; (1-z)^{\alpha}, \label{lemma2.2b} \\
&\dfrac{ f(z)}{z} \prec  \;\dfrac{1-(1-z)^{1+\alpha}}{z(1+\alpha)}, \label{lemma2.2c}\\
&\dfrac{z  f'(z)}{ f(z)} \prec   \; \dfrac{(1+\alpha)(1-z)}{1+\alpha-z}. \label{lemma2.2a}
\end{align}
\end{lemma}

\begin{lemma}\label{lemma2.3}
If $\rho\in [0,1)$ and $n \ge 2$, then the following sequences are decreasing:
\begin{align*}
    S_{1}(n,\rho)=\sum^{\infty}_{k=1}\dfrac{\rho^{k+n}}{(k+n)(k+n-1)}, \quad  S_{2}(n,\rho)=\sum ^{\infty}_{k=1} \dfrac{\rho^{k+n-1}}{k+n-1}.
\end{align*}
\begin{proof}
Note that,
\begin{align*}
S_{1}(n+1,\rho)-S_{1}(n,\rho)=& \sum_{k=1}^{\infty}\left( \dfrac{\rho}{(k+n+1)(k+n)}-\dfrac{1}{(k+n)(k+n-1)}\right)\rho^{k+n}\\
=& \sum_{k=1}^{\infty}\left( \dfrac{\rho(k+n-1)-(k+n+1)}{(k+n) \left((k+n)^2-1 \right) }\right)\rho^{k+n}\\
=& -\sum_{k=1}^{\infty}\left( \dfrac{(1-\rho)(k+n)+\rho+1}{(k+n) \left((k+n)^2-1 \right) } \right)\rho^{k+n}<0.
\end{align*}
Hence, $S_{1}(n,\rho)$ is a decreasing sequence for all $n \ge 2$. Similarly, we can show that $S_{2}(n,\rho)$ is a decreasing sequence.
\end{proof}
\end{lemma}

\begin{lemma}
Let $ f(z)=z+\sum_{k=2}^{\infty} a_k z^k \in \mathcal{G}(\alpha)$ and $\sigma_n(z;  f)=\sum_{k=n+1}^{\infty} a_k z^k$. Then, for $n \ge 2$ and $|z|=\rho <1$, the following inequalities hold:
\begin{align}
|\sigma_n(z; f)|\; \leq &\;\; \alpha\left((1-\rho)\ln(1-\rho)+\rho-\dfrac{\rho^2}{2}\right), \label{2.4} \\
|\sigma'_n(z; f)|\; \leq & \;\; -\alpha\left(\ln(1-\rho)+\rho\right), \label{2.6} \\
|\sigma''_n(z; f)|\; \leq & \;\; \dfrac{\alpha \, \rho}{1-\rho}. \label{2.7} 
\end{align}
\end{lemma}
\begin{proof}
Using Lemma \ref{lemma2.1} and Lemma \ref{lemma2.3}, we obtain
\begin{align*}
|\sigma_n(z;  f)|  &\le \; \sum^{\infty}_{k=1}\dfrac{\alpha\, \rho^{k+n}}{(k+n)(k+n-1)}=\alpha \, S_1(n,\rho) \le \alpha \, S_1(2,\rho) \\
&=  \alpha \sum_{n=3}^{\infty} \dfrac{ \rho^k}{k(k-1)}              =\alpha\left((1-\rho)\ln(1-\rho)+\rho-\dfrac{\rho^2}{2}\right).
\end{align*}
Similarly, we can find bounds on $|\sigma'_n(z; f)|$ and $|\sigma''_n(z; f)|$.
\end{proof}

\begin{lemma}{\cite[Theorem 6.2 (Rogosinski's Theorem)]{dur}}\label{rog}
Let $ f(z)=\sum_{n=1}^{\infty} a_n z^n$ and $\psi(z)=\sum_{n=1}^{\infty} c_n z^n$ be analytic in $\mathbb{D}$, and suppose $\psi \prec  f$. Then
\[ \sum_{n=1}^n |c_n|^2 \leq \sum_{n=1}^n |a_n|^2, \quad n\in \mathbb N . \]
\end{lemma}

The following result is the special case of {\cite[Lemma 5]{obra}}.
\begin{lemma}\label{lemma2.03}
Let $ f \in \mathcal{G}$ and $ s_n(z; f)$ be its nth section. Then for all $\rho \in(0, 1)$ and $n \ge 2$, we have
\[ \bigg|\dfrac{ s_n'(z; f)}{ f'(z)}-1 \bigg| \leq |z|^n\left( \dfrac{1}{n} +\left(\dfrac{ \pi}{\sqrt{6}}+1 \right) \dfrac{|z|\sqrt{2\rho-\rho^2}}{\rho^n(1-\rho)(\rho-|z|)} \right),\quad \text{for }\, |z|<\rho.\]
\end{lemma}

 \medskip
\section{Main results}

The following are our newly established results, along with their proofs.

\subsection{Radius properties for the sections}
\begin{theorem}\label{Th2.0}
Let $ f(z)=z+\sum_{k=2}^{\infty} a_k z^k \in \mathcal{G}(\alpha)$. Then $ s_n(z; f)$ is convex of order $\beta$ in the disk $|z|<\rho_{\mathcal{K}}$ for all $n\ge 2$, where $\rho_{\mathcal{K}}$ is the least positive root of $I_{\alpha,\beta}(\rho)= 0$ in $[0,1)$ for all $\alpha \in(0,1]$ and  $\beta \in[0,1)$, where $I_{\alpha,\beta}(\rho)$ is given by
\begin{align}\label{econ}
I_{\alpha,\beta}(\rho)&=(1-\rho)^{\alpha}\left(\rho(1+\alpha-\beta)-(1-\beta)\right) \\
&+ \alpha \rho\left((2-\beta)\rho-(1-\beta)\right)
 -\alpha(1-\beta)(1-\rho)\ln(1-\rho).\notag
\end{align}

\end{theorem}
\begin{proof} 
Let $ f\in \mathcal{G}(\alpha)$. Using Lemma \ref{lemma2.1}, we obtain 
\[\Re\left(1+\dfrac{z  s_2''(z; f)}{ s_2'(z; f)}\right)\geq  1- \dfrac{2|a_2| |z|}{1-2|a_2| |z|} \geq  1- \dfrac{\alpha|z|}{1-\alpha|z|}= \dfrac{1-2\alpha \rho}{1-\alpha \rho} >\beta,\]
if
\[ \rho< \dfrac{1-\beta}{\alpha(2-\beta)}:=\psi_1(\alpha,\beta).\]
Note that
\[\dfrac{\partial }{\partial \beta}\psi_1(\alpha,\beta)= -\,\dfrac{1}{\alpha(2-\beta)^2}<0.\]
Therefore
\begin{equation*}
\rho<\psi_1(\alpha,\beta)\le\psi_1(\alpha,0)=\dfrac{1}{2\alpha}:=\rho_1.
\end{equation*}
Hence, $\rho_1<1$ if $\alpha \in(1/2,1]$ and $\rho_1\ge1$ if $\alpha \in(0,1/2]$. Thus, $ s_2(z; f)$ is convex of order $\beta$ in the disk $|z|<\rho_1$ if $\alpha \in(1/2,1]$ and $ s_2(z; f)$ is convex of order $\beta$ in whole $\mathbb{D}$ if $\alpha \in(0,1/2]$. Further, using Lemma \ref{lemma2.1}, we obtain 
\begin{align*}
\Re\left(1+\dfrac{z  s_3''(z; f)}{ s_3'(z; f)}\right)& =  \Re\left(\dfrac{1+4a_2z+9a_3z^2}{1+2a_2z+3a_3z^2}\right) = 2-\Re\left(\dfrac{1-3a_3z^2}{1+2a_2z+3a_3z^2}\right) \\
 &\geq  2-\left|\dfrac{1-3a_3z^2}{1+2a_2z+3a_3z^2}\right|  \geq 2- \dfrac{1+3|a_3||z|^2}{1-2|a_2||z|-3|a_3||z|^2}\\
&\geq  2- \dfrac{2+\alpha|z|^2}{2-2\alpha|z|-\alpha|z|^2} = \dfrac{2-4\alpha \rho-3\alpha \rho^2}{2-2\alpha \rho-\alpha \rho^2}> \beta,
\end{align*}
if
\[\rho< \dfrac{\sqrt{(2-\beta)^2+2\alpha(1-\beta)(3-\beta)}-\alpha(2-\beta)}{\alpha(3-\beta)}:=\psi_2(\alpha,\beta).\]

\begin{figure}[htbp]
\centering
\includegraphics[scale=0.7]{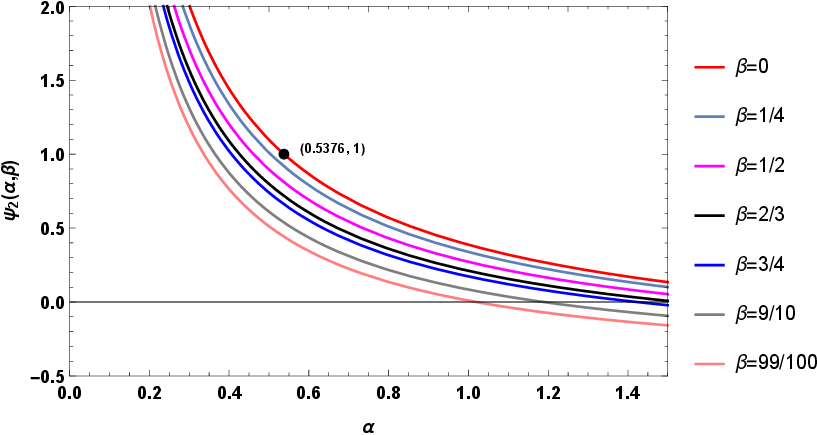}
  \caption{Graph between $\psi_2(\alpha,\beta)$ and $\alpha$, for different values of $\beta\in[0,1)$.}
  \label{gf1}
\end{figure}
Fig. \ref{gf1} shows the graphs of $\psi_2(\alpha,\beta)$ for different values of $\beta$, which shows that $\psi_2(\alpha,\beta)$ is a decreasing function of $\beta$. Hence 
\begin{equation*}
\psi_2(\alpha,\beta)\le\psi_2(\alpha,0)=\dfrac{\sqrt{4+6\alpha}-2\alpha}{3\alpha}:=\rho_2<1,
\end{equation*}
if $\alpha \in(0.5376,1]$. Note that $\rho_2\ge1$ if $\alpha \in(0,0.5376]$. Thus, $ s_3(z; f)$ is convex of order $\beta$ in the disk $|z|<\rho_2$ if $\alpha \in(0.5376,1]$, and $ s_3(z; f)$ is convex of order $\beta$ in whole $\mathbb{D}$ if $\alpha \in(0,0.5376]$.   

Next, we consider the general case for all $n \ge 2$. To prove that $ s_n(z; f)$ is convex of order $\beta$ in a disk $|z| < \rho_{\mathcal{K}}$, it is sufficient to prove that for $|z| < \rho_{\mathcal{K}},$ $ s_n'(z; f) \neq 0$ and
\begin{align}\label{covm}
\Re\left(1+\dfrac{z s_n''(z; f)}{ s_n'(z; f)} \right)& = \Re\left(1+\dfrac{z  f''(z)}{ f'(z)}+\dfrac{\dfrac{z  f''(z)}{ f'(z)} \sigma'_n(z;  f) - z \sigma''_n(z;  f)}{ f'(z)-\sigma'_n(z;  f)} \right) \\
&\ge 1 -\left| \dfrac{z  f''(z)}{ f'(z)}\right|- \dfrac{\left|\dfrac{z  f''(z)}{ f'(z)}\right| \left|\sigma'_n(z;  f)\right|+\left| z \sigma''_n(z;  f)\right|}{\left| f'(z)\right|-\left|\sigma'_n(z;  f)\right|}>\beta, \notag
\end{align}
where $\sigma_n(z;  f)= \sum ^{\infty}_{k=n+1} \alpha_k z^k$. Using \ref{lemma2.2d}, we have
\begin{align}\label{inq-convex}
\left| \dfrac{z  f''(z)}{ f'(z)}\right| \leq \dfrac{\alpha\,\rho\,(1+\rho)}{1-\rho^2}=\dfrac{\alpha\,\rho}{1-\rho}.
\end{align}
Further, using \eqref{lemma2.2b} and \eqref{2.6} for $|z|=\rho$, we obtain
\begin{equation}\label{ez}
| f'(z)|-|\sigma'_n(z;  f)| \geq (1-\rho)^{\alpha}+\alpha(\ln(1-\rho)+\rho):=\psi_3(\alpha,\rho). 
\end{equation}
Since $\dfrac{\partial \psi_3(\alpha,\rho)}{\partial \alpha}<0$ for all $\alpha \in (0,1]$, therefore
\begin{equation*}
\psi_3(\alpha,\rho)\ge \psi_3(1,\rho)=1+\ln(1-\rho),
\end{equation*}
is positive when $|z|<\rho_3,$ where $\rho_3\approx 0.6321$. 
Thus,
\begin{equation}\label{k1}
    |\sigma'_n(z;  f)| < | f'(z)|, \quad |z| < \rho_3.
\end{equation}

Note that $\dfrac{\sigma'_n(z;  f)}{ f'(z)}$ is analytic in $\mathbb{D}$ since $ f'(z)\neq 0$. Hence, if \eqref{k1} holds for $|z|=\rho_{\mathcal{K}},$ then the same inequality holds for $0<|z|<\rho_{\mathcal{K}}$ when $ s_n'(z; f) \neq 0$ for $0<|z|<\rho_{\mathcal{K}}$. Consequently, $1+\dfrac{z s_n''(z; f)}{ s_n'(z; f)}$ is analytic on $|z|\leq \rho_{\mathcal{K}}$, and if $\Re\left(1+\dfrac{z s_n''(z; f)}{ s_n'(z; f)}\right) \ge \beta$ on $|z|=\rho_{\mathcal{K}}$, then $ s_n(z; f)$ is convex of order $\beta$ in $|z|<\rho_{\mathcal{K}}$.

Now using inequalities \eqref{2.6}, \eqref{2.7}, \eqref{inq-convex}, \eqref{ez} in \eqref{covm}, we obtain
\begin{align} \label{cg(a)re}
\Re\left(1+\dfrac{z s_n''(z; f)}{ s_n'(z; f)} \right) &\ge 1- \dfrac{\alpha \rho}{1-\rho}+\dfrac{\dfrac{\alpha \rho}{1-\rho}\; \alpha(\ln(1-\rho)+\rho)-\dfrac{\alpha \rho^2}{1-\rho}}{(1-\rho)^{\alpha}+\alpha(\ln(1-\rho)+\rho)} \\
 & \ge 1- \dfrac{\alpha \rho}{1-\rho}+\dfrac{ \alpha^2 \rho\ln(1-\rho)-\alpha \rho^2(1-\alpha)}{(1-\rho)\left((1-\rho)^{\alpha}+\alpha(\ln(1-\rho)+\rho)\right)} \notag \\
&= \dfrac{(1-\rho)^{\alpha}(1-\rho(1+\alpha))+\alpha \rho(1-2\rho)+\alpha(1-\rho)\ln(1-\rho)}{(1-\rho)\left((1-\rho)^{\alpha}+\alpha(\ln(1-\rho)+\rho)\right)}\notag \\
&:=\psi_4(\alpha,\rho). \notag
\end{align}
Note that
\[\dfrac{\partial}{\partial \alpha}\psi_4(\alpha,\rho)=\dfrac{\rho(1-\rho)^{\alpha-1}\left( \alpha \ln(1-\rho)((1-\alpha)\rho-\alpha\ln(1-\rho))-(1-\rho)^{\alpha} \right)}{(\alpha(\ln(1-\rho)+\rho)+(1-\rho)^{\alpha})^2}<0.\]
Thus
\[\psi_4(\alpha,\rho)\ge \psi_4(1,\rho)=\dfrac{1-2\rho+(1-\rho)\ln(1-\rho)}{(1-\rho)(1+\ln(1-\rho))}>0,\]
for $|z|<\rho_4$, where $\rho_4 \approx 0.3578$. Since,
\[\min \{ \rho_1, \rho_2, \rho_3, \rho_4\}=\rho_4=\rho_{\mathcal{K}},\] 
therefore
\[\Re\left(1+\dfrac{z s_n''(z; f)}{ s_n'(z; f)} \right)\ge \psi_4(\alpha,\rho) > \beta,\quad |z|<\rho_{\mathcal{K}}. \]
That is, $ s_n(z; f)$ is convex of order $\beta$ in the disk $|z|<\rho_{\mathcal{K}}$ for all $n\ge 2$, where $\rho_{\mathcal{K}}$ is the least positive root of $I_{\alpha,\beta}(\rho)= 0$ in $[0,1)$ for all $\alpha \in(0,1]$ and  $\beta \in[0,1)$, where $I_{\alpha,\beta}(\rho)$ is given by \eqref{econ}. This completes the proof.
\end{proof}

Using Theorem \ref{Th2.0}, we yield the following result:
\begin{corollary}\label{c1}
Let $ f(z)=z+\sum_{k=2}^{\infty} a_k z^k \in \mathcal{G}$. Then $ s_n(z; f)$ is convex in $|z|<\rho_{\mathcal{K}_1}\approx 0.3578$ for all $n\ge 2$, where $\rho_{\mathcal{K}_1}$ is the unique root of $I(\rho)=0$ in $[0,1)$, where $I(\rho)$ is given by
\[I(\rho)=1-2\rho+(1-\rho)\ln(1-\rho).\]
\end{corollary}
\begin{remark}
    The Corollary \ref{c1} improved the result of Maharana et~al. \cite[Theorem 1]{mah}, which states that every section $ s_n(z; f)$ of $ f(z) \in \mathcal{G}$ is convex in the disk $|z|<1/2$.
\end{remark}

 \medskip
\begin{theorem}\label{Th2.1}
Let $ f(z)=z+\sum_{k=2}^{\infty} a_k z^k \in \mathcal{G}(\alpha)$. Then $ s_n(z; f)$ is starlike of order $\beta$ in the disk $|z|<\rho_{\mathcal{S}^*}$ for all $n\ge 2$, where $\rho_{\mathcal{S}^*}$ is the least positive root of $J_{\alpha, \beta}(\rho)=0$ in $[0,1)$ for all $\alpha \in(0,1]$ and  $\beta \in[0,1)$, where $J_{\alpha,\beta}(\rho)$ is given by
\begin{align}\label{2.1}
J_{\alpha, \beta}(\rho)&=(1-(1-\rho)^{1+\alpha})((1+\alpha)(1-\beta-\rho)+\beta \rho)\\
&+\alpha(1+\alpha)(1+\alpha-\rho)\left(\left( (3-\beta)\rho-(2-\beta)\right) \ln(1-\rho)-(2-\beta)\rho+\left(2-\dfrac{\beta}{2} \right)\rho^2\right). \notag
\end{align}
\end{theorem}
\begin{proof} 
Let $ f\in \mathcal{G}(\alpha)$. Using Lemma \ref{lemma2.1}, we obtain 
\[\Re\left(\dfrac{z s_2'(z; f)}{ s_2(z; f)}\right)=\Re\left(\dfrac{z+2 a_2 z^2}{z+a_2 z^2}\right)\geq 1- \dfrac{|a_2||z|}{1-|a_2| |z|} \geq  \dfrac{2(1- \alpha \rho)}{2- \alpha \rho}>\beta,\]
if
\[\rho< \dfrac{2(1-\beta)}{\alpha(2-\beta)}:=\psi_5(\alpha,\beta).\]
Note that
\[\dfrac{\partial }{\partial \beta}\psi_5(\alpha,\beta)=-\,\dfrac{2}{\alpha(2-\beta)^2}<0.\]
Therefore
\begin{equation*}
\rho<\psi_5(\alpha,\beta)\le\psi_5(\alpha,0)=\dfrac{1}{\alpha}:=\rho_5.
\end{equation*}
Note that $\rho_5\ge 1$ for all $\alpha \in(0,1]$. Thus, $ s_2(z; f)$ is startlike of order $\beta$ in $\mathbb D$ for all $\alpha \in(0,1]$. Further using Lemma \ref{lemma2.1}, we obtain 
\begin{align*}
\Re\left(\dfrac{z  s_3'(z; f)}{ s_3(z; f)}\right)=\Re\left(\dfrac{z+2a_2 z^2 + 3a_3 z^3}{z+a_2 z^2+a_3 z^3}\right)
\geq & 1- \dfrac{|a_2||z|+2|a_3||z|^2}{1-|a_2||z|-|a_3||z|^2}\\
\geq & \dfrac{6-6\alpha \rho-3\alpha \rho^2}{6-3\alpha \rho-\alpha \rho^2} > \beta,
\end{align*}
if
\[\rho< \dfrac{\sqrt{9(2-\beta)^2\alpha^2+24\alpha (1-\beta)(3-\beta)}-3(2-\beta)\alpha}{2 \alpha(3-\beta)}:=\psi_6(\alpha,\beta).\]

\begin{figure}[htbp]
\centering
\includegraphics[scale=0.7]{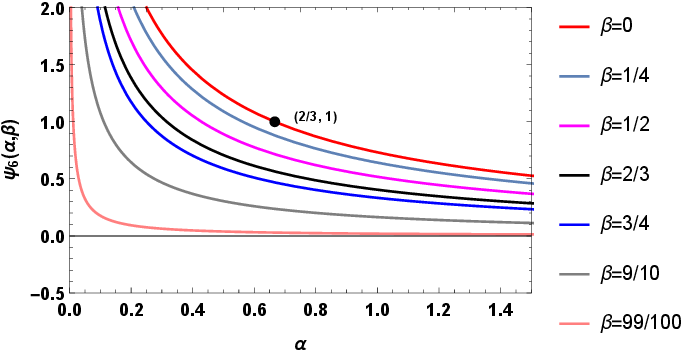}
  \caption{Graph between $\psi_6(\alpha,\beta)$ and $\alpha$, for different values of $\beta\in[0,1)$.}
  \label{gf2}
\end{figure}
Fig. \ref{gf2} shows the graphs of $\psi_6(\alpha,\beta)$ for different values of $\beta$, which shows that $\psi_6(\alpha,\beta)$ is a decreasing function of $\beta$, therefore
\begin{equation*}
\psi_6(\alpha,\beta)\le\psi_6(\alpha,0)=\sqrt{1+\dfrac{2}{\alpha}}-1:=\rho_6<1,
\end{equation*}
if $\alpha \in(2/3,1]$ and $\rho_6\ge 1$ if $\alpha \in(0,2/3]$. Thus, $ s_3(z; f)$ is starlike of order $\beta$ in the disk $|z|<\rho_6$ if $\alpha \in(2/3,1]$, and $ s_3(z; f)$ is starlike of order $\beta$ in whole $\mathbb D$ if $\alpha \in(0,2/3]$.

Next, we consider the general case for $n \ge 2$. To prove that $ s_n(z; f)$ is starlike of order $\beta$ in $|z|< \rho_{\mathcal{S}^*}$, it is sufficient to prove that for $0 < |z|< \rho_{\mathcal{S}^*}$, $ s_n(z; f)\neq 0$ and
\begin{align}\label{2.2}
\displaystyle{\Re\left(\dfrac{z s_n'(z; f)}{ s_n(z; f)}\right)}=& \; \Re\left(\dfrac{z  f'(z)}{ f(z)}+\dfrac{\dfrac{z  f'(z)}{ f(z)}\sigma_n(z; f)-z\sigma'_n(z; f)}{ f(z)-\sigma_n(z; f)}\right) \\
>& \; \Re\left(\dfrac{z  f'(z)}{ f(z)} \right)- \dfrac{\displaystyle{\Big|\dfrac{z  f'(z)}{ f(z)}\Big| |\sigma_n(z; f)|+|z\sigma'_n(z; f)|}}{| f(z)|-|\sigma_n(z; f)|}>\beta, \notag
\end{align}
where $ \sigma_n(z; f)= \sum ^{\infty}_{k=n+1} a_k z^k$.
From \eqref{lemma2.2a}, we get
\begin{equation} \label{2.8}
     \dfrac{z  f'(z)}{ f(z)}=\dfrac{(1+\alpha)(1-w(z))}{1+\alpha-w(z)}=1-\dfrac{\alpha\, w(z)}{1+\alpha-w(z)}.
\end{equation}
Now, for $|z|=\rho<1$, we obtain
\begin{align}\label{2.9}
\Re \left( \dfrac{z  f'(z)}{ f(z)} \right)\ge 1-\bigg| \dfrac{\alpha\, w(z)}{1+\alpha-w(z)} \bigg| \ge \dfrac{(1+\alpha)(1-\rho)}{1+\alpha-\rho},
\end{align}
and
\begin{equation}\label{2.10}
\bigg| \dfrac{z  f'(z)}{ f(z)} \bigg| \leq 1+\dfrac{\alpha\, |w(z)|}{1+\alpha-|w(z)|} \leq \dfrac{1+\alpha -(1-\alpha) \rho}{1+\alpha-\rho}.
\end{equation}
Further using \eqref{lemma2.2b}, we obtain
\begin{equation}\label{2.11}
| f(z)|=\int_{|z|=\rho}| f'(\zeta)||d\zeta|\ge \dfrac{1-(1-\rho)^{1+\alpha}}{1+\alpha}.
\end{equation}
To show that $ s_n(z; f)$ is starlike of order $\beta$, it is sufficient to show that
\begin{equation}\label{2.11.1}
    |\sigma_n(z;  f)| < | f(z)|, \quad |z| < \rho_{\mathcal{S}^*},
\end{equation}
and the inequality \eqref{2.2} holds for $|z| < \rho_{\mathcal{S}^*}$.

In fact, $\dfrac{\sigma_n(z;  f)}{ f(z)}$ is analytic in $\mathbb{D}$. Hence, if \eqref{2.11.1} holds for $|z|=\rho_{\mathcal{S}^*},$ then the same inequality holds for $0<|z|<\rho_{\mathcal{S}^*}$ when $ s_n(z; f) \neq 0$ for $0<|z|<\rho_{\mathcal{S}^*}$. Consequently $\dfrac{z s_n'(z; f)}{ s_n(z; f)}$ is analytic on $|z|\leq \rho_{\mathcal{S}^*}$ and if $\Re\left(\dfrac{z s_n'(z; f)}{ s_n(z; f)}\right) \geq 0$ on $|z|=\rho_{\mathcal{S}^*}$, then $ s_n(z; f)$ is starlike in $|z|<\rho_{\mathcal{S}^*}$. But the inequality $\Re\left(\dfrac{z s_n'(z; f)}{ s_n(z; f)}\right) \geq 0$ on $|z|=\rho_{\mathcal{S}^*}$ follows from \eqref{2.2}. Using \eqref{2.11} and \eqref{2.4}, we obtain
\[
| f(z)|-|\sigma_n(z;  f)| >  \dfrac{1-(1-\rho)^{1+\alpha}}{1+\alpha}-\alpha\left((1-\rho)\ln(1-\rho)+\rho-\dfrac{\rho^2}{2}\right):=\psi_7(\alpha,\rho). 
\]
Since $\dfrac{\partial }{\partial \alpha}\psi_7(\alpha,\rho)<0$ for all $\alpha \in (0,1]$, therefore
\[\psi_7(\alpha,\rho)\ge \psi_7(1,\rho)=-(1-\rho)\ln(1-\rho):=\rho_7,\]
which is positive when $|z|<\rho_7,$ where $\rho_7=1$. Hence \eqref{2.11.1} holds when $|z|<\rho_7$, that is in whole $\mathbb{D}$.

By using the inequalities \eqref{2.4},\eqref{2.6}, \eqref{2.9}, \eqref{2.10}, and \eqref{2.11.1}, we obtain 
\begin{align}\label{sg(a)re}
& \Re\left(\dfrac{z s_n'(z; f)}{ s_n(z; f)}\right)  \\
\ge &\; \dfrac{(1+\alpha)(1-\rho)}{1+\alpha-\rho} - \dfrac{\dfrac{\alpha(1+\alpha-(1-\alpha)\rho)}{1+\alpha-\rho}\left((1-\rho)\ln(1-\rho)+\rho-\dfrac{\rho^2}{2}\right)-\alpha \, \rho(\ln(1-\rho)+\rho)}{\dfrac{1-(1-\rho)^{1+\alpha}}{1+\alpha}-\alpha\left((1-\rho)\ln(1-\rho)+\rho-\dfrac{\rho^2}{2}\right)} \notag \\
= &\; \dfrac{(1-\rho)(1-(1-\rho)^{1+\alpha})+\alpha(1+\alpha-\rho)\left((3\rho-2)\ln(1-\rho)-2\rho+2\rho^2\right)}{(1+\alpha-\rho)\left( \dfrac{1-(1-\rho)^{1+\alpha}}{1+\alpha}-\alpha\left((1-\rho)\ln(1-\rho)+\rho-\dfrac{\rho^2}{2}\right)\right)} :=\psi_8(\alpha,\rho). \notag
\end{align}
Note that,
\begin{align*}
\dfrac{\partial}{\partial \alpha}\psi_8(\alpha,\rho)= \dfrac{\eta_0(\alpha,\rho)\,\eta_1(\alpha,\rho)\,\eta_2(\alpha,\rho)-\left(\eta_0(\alpha,\rho)\,\eta_3(\alpha,\rho)+\eta_1(\alpha,\rho)\right)\eta_4(\alpha,\rho)}{\eta_0^2(\alpha,\rho)\,\eta_1^2(\alpha,\rho)}<0,
\end{align*}
in the interval $(0,1)$ where,
\begin{align*}
\eta_0(\alpha,\rho)=& \, (1+\alpha-\rho) ,\\
\eta_1(\alpha,\rho)=& \, \dfrac{1-(1-\rho)^{1+\alpha}}{1+\alpha}-\alpha\left((1-\rho)\ln(1-\rho)+\rho-\dfrac{\rho^2}{2}\right) ,\\
\eta_2(\alpha,\rho)=&  -2(1+2\alpha-\rho)(1-\rho)\rho \\
& \; \quad -\left(\alpha(4-6\rho)+(1-\rho)\left(2+(1-\rho)^{\alpha}-(3+(1-\rho)^{\alpha})\rho \right) \right)\ln(1-\rho) , \\
\eta_3(\alpha,\rho)=&  -\dfrac{1-(1-\rho)^{1+\alpha}}{(1+\alpha)^2}-\rho+\dfrac{\rho^2}{2}- \dfrac{(1-\rho)^{1+\alpha}\ln(1-\rho)}{1+\alpha} -(1-\rho)\ln(1-\rho) ,\quad \text{and}\\
\eta_4(\alpha,\rho)=& \,(1-(1-\rho)^{1+\alpha})(1-\rho)-\alpha(1+\alpha-\rho)\left(2(1-\rho)\rho+(2-3\rho)\ln(1-\rho)\right).
\end{align*} 
Therefore
\[\psi_8(\alpha,\rho)\ge \psi_8(1,\rho)=\dfrac{\rho-\rho^2+(2-3\rho)\ln(1-\rho)}{(1-\rho)\ln(1-\rho)}>0,\]
for $|z|<\rho_8$, where $\rho_8 \approx 0.5698$. Note that
\[\min \{ \rho_5, \rho_6, \rho_7, \rho_8\}=\rho_8=\rho_{\mathcal{S}^*}.\] 
Therefore
\[\Re\left(\dfrac{z s_n'(z; f)}{ s_n(z; f)} \right)\ge \psi_8(\alpha,\rho) > \beta,\quad |z|<\rho_{\mathcal{S}^*}. \]
That is, $ s_n(z; f)$ is starlike of order $\beta$ in the disk $|z|<\rho_{\mathcal{S}^*}$ for all $n\ge 2$, where $\rho_{\mathcal{S}^*}$ is the least positive root of $J_{\alpha, \beta}(\rho)=0$ in $[0,1)$ for all $\alpha \in(0,1]$ and  $\beta \in[0,1)$, where $J_{\alpha,\beta}(\rho)$ is given by \eqref{2.1}. This completes the proof.
\end{proof}

Now, using Theorem \ref{Th2.1}, we yield the following result:
\begin{corollary}\label{c2}
Let $ f(z)=z+\sum_{k=2}^{\infty} a_k z^k \in \mathcal{G}$. Then $ s_n(z; f)$ is starlike in $|z|<\rho_{\mathcal{S}^*_1}$ for all $n\ge 2$, where $\rho_{\mathcal{S}^*_1}\approx 0.5698$ is the unique root of $J(\rho)=0$ in $[0,1)$, where $J(\rho)$ is given by
\[J(\rho)=\rho-\rho^2+(2-3\rho)\ln(1-\rho).\]
\end{corollary}

 \medskip
\begin{theorem}\label{Th2.2}
Let $ f(z)=z+\sum_{k=2}^{\infty} a_k z^k \in \mathcal{G}(\alpha)$. Then $ s_n(z; f)$ is close-to-convex of order $\beta$ in the disk $|z|<\rho_{\mathcal{C}}$ for all $n\ge 2$, where $\rho_{\mathcal{C}}$ is the least positive root of $K_{\alpha, \beta}(\rho)= 0$ in $[0,1)$ for all $\alpha \in(0,1]$ and  $\beta \in[0,1)$, where $K_{\alpha,\beta}(\rho)$ is given by
\begin{align}\label{2.12}
K_{\alpha, \beta}(\rho)=(1-\rho)^{\alpha}+\alpha(\ln(1-\rho)+\rho)-\beta.
\end{align}
\end{theorem}

\begin{proof} 
Let $ f\in \mathcal{G}(\alpha)$. Using Lemma \ref{lemma2.1},  we obtain 
\[\Re\left(  s_2'(z; f)\right)\geq 1-2|a_2||z| \geq 1-\alpha|z|>\beta, \]
if
\[\rho<\dfrac{1-\beta}{\alpha}:=\psi_9(\alpha,\beta).\]
Note that
\[\dfrac{\partial }{\partial \beta}\psi_9(\alpha,\beta)=-\,\dfrac{1}{\alpha}<0.\]
Therefore
\[\rho<\psi_9(\alpha,\beta)\le\psi_9(\alpha,0)=\dfrac{1}{\alpha}:=\rho_9.\]
Note that $\rho_9\ge 1$ for all $\alpha \in(0,1]$. Thus, $ s_2(z; f)$ is close-to-convex of order $\beta$ in $\mathbb D$, for all $\alpha \in(0,1]$. Further using Lemma \ref{lemma2.1}, we obtain
\[\Re\left(  s_3'(z; f)\right)\geq 1-2|a_2||z|-3|a_3||z|^2 \geq 1-\alpha|z|-\dfrac{\alpha}{2}|z|^2>\beta, \]
if
\[\rho<\sqrt{1+\dfrac{2(1-\beta)}{\alpha}}-1:=\psi_{10}(\alpha,\beta).\]
Note that
\[\dfrac{\partial}{\partial \beta} \psi_{10}(\alpha,\beta)=\dfrac{-1}{\alpha \sqrt{1+\dfrac{2(1-\beta)}{\alpha}}}<0.\]
Therefore 
\[\rho<\psi_{10}(\alpha,\beta)\le\psi_{10}(\alpha,0)=\sqrt{1+\dfrac{2}{\alpha}}-1:=\rho_{10}<1,\]
if $\alpha \in (2/3,1]$, and $\rho_{10}\ge1$ if $\alpha \in (0,2/3]$. Thus, $ s_3(z; f)$ is close-to-convex of order $\beta$ in the disk $|z|<\rho_{10}$ if $\alpha \in(2/3,1]$ and $ s_3(z; f)$ is close-to-convex of order $\beta$ in whole $\mathbb D$ if $\alpha \in(0,2/3]$.

Next, we consider the general case for all $n \ge 2$. Using \eqref{lemma2.2b} and \eqref{2.6}, we obtain
\begin{align}\label{ctcg(a)re} 
    \Re( s_n'(z; f)) = \Re( f'(z)-\sigma'_n(z; f))
&\geq (1-\rho)^{\alpha} - |\sigma'_n(z; f)| \\
&=(1-\rho)^{\alpha}+\alpha(\ln(1-\rho)+\rho):=\psi_{11}(\alpha,\rho). \notag
\end{align}
Note that
\begin{align*}
\dfrac{\partial }{\partial \alpha}\psi_{11}(\alpha,\rho)=\rho+\left(1+(1-\rho)^{\alpha}\right)\ln(1-\rho)<0,
\end{align*}
in the interval $(0,1)$, for all $\alpha \in (0,1]$. Therefore
\[\psi_{11}(\alpha,\rho)\ge \psi_{11}(1,\rho)=1+\ln(1-\rho)>0,\]
for $|z|<\rho_{11}$, where $\rho_{11}\approx 0.6321$. Thus
\[\min \{ \rho_9, \rho_{10}, \rho_{11} \}=\rho_{11}=\rho_{\mathcal{C}}.\] 
Now, using \eqref{ctcg(a)re}, we obtain
\[\Re\left( s_n'(z; f) \right)\ge \psi_{11}(\alpha,\rho) > \beta,\quad |z|<\rho_{\mathcal{C}}. \]
Therefore, $ s_n(z; f)$ is close-to-convex of order $\beta$ in $|z|<\rho_{\mathcal{C}}$ for all $n\ge 2$, where $\rho_{\mathcal{C}}$ is the least positive root of $K_{\alpha, \beta}(\rho)= 0$ in $[0,1)$ for all $\alpha \in(0,1]$ and  $\beta \in[0,1)$, where $K_{\alpha,\beta}(\rho)$ is given by \eqref{2.12}. This completes the proof.
\end{proof}

Further, using Theorem \ref{Th2.2}, we yield the following result:
\begin{corollary}\label{c3}
Let $ f(z)=z+\sum_{k=2}^{\infty} a_k z^k \in \mathcal{G}$. Then $ s_n(z; f)$ is close-to-convex in $|z|<\rho_{\mathcal{C}_1}$ for all $n\ge 2$, where $\rho_{\mathcal{C}_1}\approx 0.6321$ is the unique root of $K(\rho)=0$ in $[0,1)$, where $K(\rho)$ is given by
\[K(\rho)=1+\ln(1-\rho).\]
\end{corollary}

 \medskip
\subsection{Inequalities and properties of the sections}
\begin{theorem}\label{Th2.5}
Let $ f \in \mathcal{G}$. Then, for each $n\ge 2$, we have
\begin{equation}\label{2.16}
\bigg|\dfrac{ s_n(z; f)}{ f(z)}-1 \bigg| \leq |z|^n\left( \dfrac{1}{n(n+1)} + \dfrac{2|z|}{\sqrt{3}(1-|z|)}\right),\quad |z|=\rho<1.
\end{equation}
\end{theorem}
\begin{proof}
Let $ f \in \mathcal{G}$. Since $ f \in \mathcal{G}$ is univalent in $\mathbb{D}$, $z/ f$ can be represented in the form 
\[\dfrac{z}{ f(z)}=1+b_1z+b_2z^2+\dots \quad,\]
where $b_1=-a_2$, $b_2=a_{2}^2+a_3$ and so on, for some complex coefficients $b_n$, $n\geq 1$.
Now considering the identity $\dfrac{ f(z)}{z}\dfrac{z}{ f(z)}\equiv 1 $, we have
\[(1+a_2z+a_3z^2+\dots)(1+b_1z+b_2z^2+\dots)= 1, \quad z \in \mathbb{D}. \]
From which, we obtain
\begin{equation}\label{2.17}
 a_m + \sum_{k=1}^{m-1}b_ka_{m-k}=0,
\end{equation}
where $m\geq 2$ and $a_1=1$.
Now using \eqref{2.17} and the series representation of section $ s_n(z; f)$, we obtain
\begin{align}\label{2.18}
\dfrac{ s_n(z; f)}{ f(z)}=&(1+a_2 z+\dots+a_n z^{n-1})(1+b_1z+b_2z^2+\dots) \\
\equiv & 1+e_nz^n+e_{n+1}z^{n+1}+\dots \quad, \notag
\end{align}
where 
\begin{align}\label{2.19}
e_n=&\; a_n b_1+ a_{n-1} b_2+\dots+a_2 b_{n-1}+ b_n,\\
e_{n+1}=&\; a_n b_2+ a_{n-1}b_3+\dots+a_2 b_n + b_{n+1},\notag \\
e_{n+2}=&\; a_n b_3+a_{n-1} b_4+\dots+a_2 b_{n+1}+ b_{n+2},\quad \;\;(a_1=1) \notag
\end{align}
and so on.
Putting $m=n+1$ in \eqref{2.17} and comparing with \eqref{2.19}, we obtain
\begin{equation*}
 e_n=-a_{n+1}.
\end{equation*}
Now, using Lemma \ref{lemma2.1}, we obtain
\begin{equation}\label{2.20}
|e_n|=(n+1)|a_{n+1}|\leq \dfrac{1}{n(n+1)}.
\end{equation}
Observing the series $e_{n+1}$, $e_{n+2}$, and so on, for $m\geq n+1$, we obtain
\begin{equation*}
e_m=a_n b_{m-n+1}+a_{n-1} b_{m-n+2}+\dots+a_2 b_{m-1}+ b_m.
\end{equation*}
Using Lemma \ref{lemma2.1}, we obtain
\begin{align}\label{2.21}
|e_m|\leq & |a_n|| b_{m-n+1}|+|a_{n-1}|| b_{m-n+2}|+\dots+|a_2|| b_{m-1}|+| b_m|  \\
\leq & \dfrac{1}{n(n-1)}| b_{m-n+1}|+\dfrac{1}{(n-1)(n-2)}| b_{m-n+2}|+\dots+\dfrac{1}{2}| b_{m-1}|+|d_m| \notag \\
= & \sum_{k=2}^{n} \dfrac{1}{k(k-1)}| b_{m-k+1}|+|b_m|.\notag
\end{align}
From \eqref{lemma2.2c} for $\alpha=1$, we obtain
$\dfrac{z}{ f(z)}\prec \dfrac{1}{1-z/2}=1+\sum_{k=1}^{\infty} \dfrac{1}{2^k}z^k$.
Now, using Lemma \ref{rog}, we obtain
\[ \sum_{k=1}^{n}|b_k|^2 \leq \sum_{k=1}^{n} \dfrac{1}{2^{2k}}=\dfrac{1}{3}\left(1-\dfrac{1}{4^n} \right). \]
It follows 
$$\sum_{k=1}^{n}|b_k|^2 \leq \sum_{k=1}^{\infty}|b_k|^2 \leq \dfrac{1}{3}.$$ Therefore
\begin{equation}\label{2.22}
|b_k|^2 \leq \dfrac{1}{3} \Rightarrow |b_k| \leq \dfrac{1}{\sqrt{3}},
\end{equation}
for each $k \in \mathbb{N}$. Now using \eqref{2.21} and \eqref{2.22} for $m\geq n+1$, we obtain
\begin{equation}\label{2.23}
|e_m|\leq  \dfrac{1}{\sqrt{3}} \sum_{k=2}^{n} \dfrac{1}{k(k-1)} + \dfrac{1}{\sqrt{3}}= \dfrac{2}{\sqrt{3}}.
\end{equation}
Again using \eqref{2.20} and \eqref{2.23} in \eqref{2.18}, we obtain
\begin{align}
\bigg| \dfrac{ s_n(z; f)}{ f(z)}-1 \bigg| =& |e_nz^n+e_{n+1}z^{n+1}+\dots|  \\
\leq & |e_n||z^n|+|e_{n+1}||z|^{n+1}+\dots \notag \\
\leq & |z|^n\left(\dfrac{1}{n(n+1)}+\dfrac{2}{\sqrt{3}}|z|+\dfrac{2}{\sqrt{3}}|z|^2\dots\right) \notag \\
=& |z|^n \left( \dfrac{1}{n(n+1)}+\dfrac{2|z|}{\sqrt{3}(1-|z|)}\right),\notag
\end{align}
for $|z|=\rho<1$ and for $n \ge 2$. This completes the proof of the theorem.
\end{proof}

 \medskip
\begin{remark}
    If $ f \in \mathcal{G}$, then in view of Corollary \ref{c3}, we observe that the radius of the disk for which
    \[\Re( s_n'(z; f))\ge1-\sum_{k=2}^{n}k|a_k||z|^{k-1}>0,\]
    is decreasing as $n$ is increasing, but it remains constant for $n\geq 17$, which can be seen in Fig. \ref{fig:ctcgp2}. The following result is devoted to establishing this fact.
\end{remark}

\begin{figure}[htbp]
\centering
\includegraphics[scale=0.6]{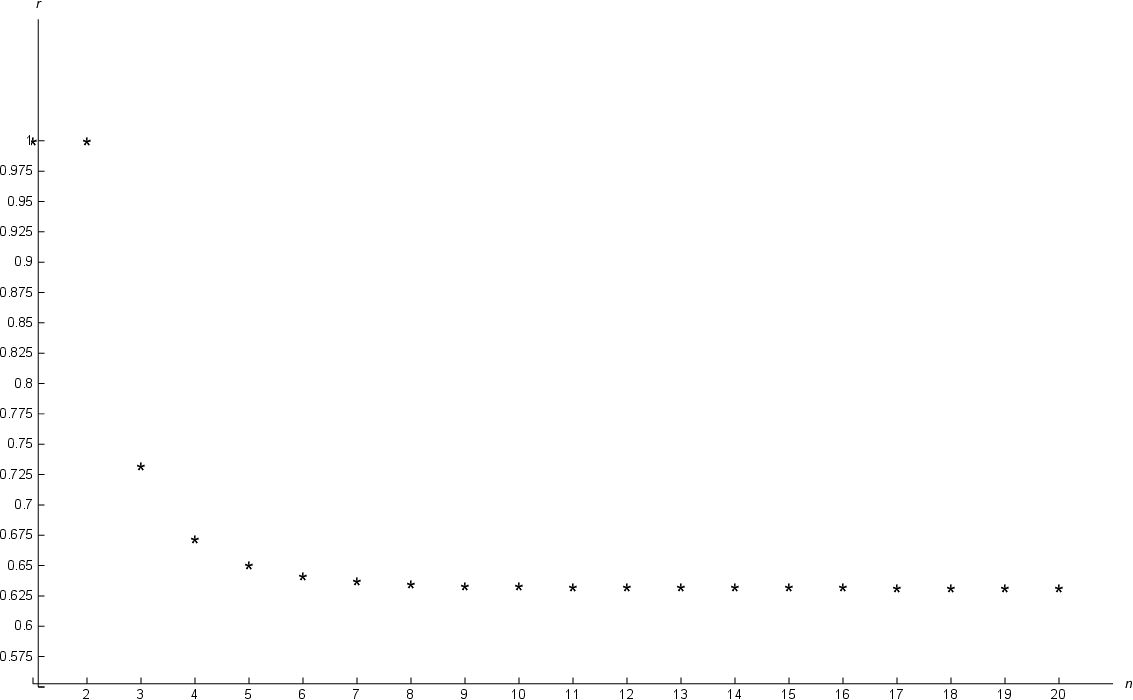}
  \caption{Graph between $n$ and approximate value of radius of disk in which $\sum_{k=2}^{n}k|a_k||z|^{k-1}\le 1$.}
  \label{fig:ctcgp2}
\end{figure}

\begin{theorem}\label{Th2.6}
Let $ f \in \mathcal{G}$. Then $\Re( s_n'(z; f))>0$ in $|z|\leq 0.6321$ for $n\geq 17$.
\end{theorem}
\begin{proof}
If $ f \in \mathcal{G}$, then $ f'(z) \prec 1-z$. Therefore, using the maximum modulus principle for $|z|\leq 0.6321$, we obtain
\begin{equation}\label{2.24}
\underset{|z|= 0.6321}{\max} |\arg  f'(z)|\leq \sin^{-1}(0.6321)< 39.206^{\circ}.
\end{equation}
Using Lemma \ref{lemma2.03} for $|z|= 0.6321$ and $\rho=4/5$, we obtain
\begin{equation*}
\bigg|\dfrac{ s_n'(z; f)}{ f'(z)}-1 \bigg| \leq C_n,\quad |z|< 0.6321,
\end{equation*}
where $C_n=(0.6321)^n\left( \dfrac{1}{n} + \dfrac{5^n \times 42.09795334 }{4^n} \right)$.
Using the maximum modulus principle, it follows that
\begin{equation}\label{2.25}
\underset{|z|= 0.6321}{\max} \bigg|\arg \dfrac{ s_n'(z; f)}{ f'(z)}\bigg|\leq \sin^{-1}(C_n),\quad |z|< 0.6321.
\end{equation}
Further, using \eqref{2.24} and \eqref{2.25}, we obtain
$$|\arg  s_n'(z; f)|\leq |\arg  f'(z)| +  \bigg|\arg \dfrac{ s_n'(z; f)}{ f'(z)}\bigg|< 39.206^{\circ} + \sin^{-1}(C_n)<\dfrac{\pi}{2}.$$
This holds true if $\sin^{-1}(C_n)< 50.794^{\circ}$. Now the result follows, as the last inequality holds true for $n\geq 17$.
\end{proof}

\begin{remark}
     Theorem \ref{Th2.6} improved the result of Obradovic et~al. \cite[Corollary 4]{obra}, which states that $\Re( s_n'(z; f))>0$ in $|z|<1/2$ for $n\geq 11$.
 \end{remark}

 \medskip
\begin{remark}
    If $ f \in \mathcal{G}$, then in view of Corollary \ref{c2}, we observe that the radius of the disk for which
    \[\Re \left( \dfrac{z s_n'(z; f)}{ s_n(z; f)} \right)\ge 1-\dfrac{\sum_{k=2}^{n}(k-1)|a_k||z|^{k-1}}{1-\sum_{k=2}^{n}|a_k||z|^{k-1}} >0,\]
    is decreasing as $n$ is increasing, but remains constant for $n\geq 17$, which can be seen in Fig. \ref{fig:ctcgp2}. But we observe that it can be improved. The following result is devoted to establishing this fact.
\end{remark}

\begin{theorem}\label{Th2.7}
Let $ f \in \mathcal{G}$. Then $ s_n(z; f)$ is starlike in $|z|\leq 0.5698$ for $n\geq 10$.
\end{theorem}
\begin{proof}
Let $ f\in \mathcal{G}$. Using Lemma \ref{lemma2.03} for $|z|= 0.5698$ and $\rho=4/5$, we obtain
\begin{equation*}
\bigg|\dfrac{ s_n'(z; f)}{ f'(z)}-1 \bigg| \leq E_n,\quad |z|\leq 0.5698.
\end{equation*}
where $E_n=(0.5698)^n\left( \dfrac{1}{n} + \dfrac{5^n \times 27.67852953}{4^n} \right)$. Using the maximum modulus principle, it follows that
\begin{equation}\label{2.26}
\underset{|z|= 0.5698}{\max} \bigg|\arg \dfrac{ s_n'(z; f)}{ f'(z)}\bigg|\leq \sin^{-1}(E_n),\quad |z|\leq 0.5698.
\end{equation}
Using \eqref{lemma2.2a} for $\alpha=1$, after a computation, we obtain
\begin{equation}\label{2.27}
\bigg|\arg \dfrac{z  f'(z)}{ f(z)}\bigg|\leq \sin^{-1}\left( \dfrac{0.5698}{2-(0.5698)^2} \right)<19.88378^{\circ},\quad |z|\leq 0.5698.
\end{equation}
Now, using Theorem \ref{Th2.5} for $|z|= 0.5698$, we obtain
\begin{equation*}
\bigg|\dfrac{ s_n(z; f)}{ f(z)}-1 \bigg| \leq F_n,\quad |z|\leq 0.5698.
\end{equation*}
where $F_n=(0.5698)^n\left( \dfrac{1}{n(n+1)} + 1.529401131 \right)$. Using the maximum modulus principle, it follows that
\begin{equation}\label{2.28}
\underset{|z|= 0.5698}{\max} \bigg|\arg \dfrac{ s_n(z; f)}{ f(z)}\bigg|\leq \sin^{-1}(F_n),\quad |z|\leq 0.5698.
\end{equation}
Combining \eqref{2.26}, \eqref{2.27}, and \eqref{2.28}, we obtain
\begin{align*}
\bigg|\arg \dfrac{z  s_n'(z; f)}{ s_n(z; f)}\bigg| \leq & \bigg|\arg \dfrac{ s_n'(z; f)}{ f'(z)}\bigg| + \bigg|\arg \dfrac{z  f'(z)}{ f(z)}\bigg| + \bigg|\arg \dfrac{ f(z)}{ s_n(z; f)}\bigg| \\
<& \sin^{-1}(E_n) + 19.88378^{\circ} + \sin^{-1}(F_n)< \dfrac{\pi}{2} ,
\end{align*}
which holds true if $\sin^{-1}(E_n) + \sin^{-1}(F_n)< 70.116271^{\circ}$. The last result holds true for $n\geq 10$.
\end{proof}

\begin{remark}
     Theorem \ref{Th2.7} improved the result of Obradovic et~al. \cite[Theorem 6]{obra}, which states that $ s_n(z; f)$ of $ f \in \mathcal{G}$ is starlike  in $|z|<1/2$ for $n\geq 11$.
 \end{remark}


\begin{thebibliography}{99}

\bibitem{dur} Duren, P.L.: Univalent functions (1983), Grundlehen der Mathematischen Wissenschaften, New York: Springer- Verlag. \textbf{259}

\bibitem{um} Umezawa, T.: Analytic functions convex in one direction, J. Math. Soc. Japan. \textbf{4}, 194–202 (1952)

\bibitem{oz} Ozaki, S.: On the theory of multivalent functions. II, Sci. Rep. Tokyo Bunrika Daigaku Sect. A, \textbf{4}, 45-87 (1941)

\bibitem{sig} Singh, R., Singh, S.: Some sufficient conditions for univalence and starlikeness, Colloq. Math., \textbf{47}, 309–314 (1982)

\bibitem{szego} Szego, G.: Zur Theorie der schlichten Abbildungen, Math. Ann. \textbf{100}, 188-211 (1928)

\bibitem{robertson} Robertson, M.S.: The Partial Sums of Multivalently Star-Like Functions, Ann. of Math. (2). \textbf{42}, 829-838 (1941)

\bibitem{bshouty} Bshouty, D., Hengartner, W.: Criteria for extremality of the koebe mapping, Proc. Amer. Math. Soc. \textbf{111}, 403-411 (1991)

\bibitem{obra2} Obradovi\'{c}, M., Ponnusamy, S.: Injectivity and starlikeness of sections of a class of univalent functions, Contemp. Math. \textbf{591}, 195–203 (2013)

\bibitem{kar} Kargar, R., Pascu, N.R., Ebadian, A.: Locally univalent approximations of analytic functions, J. Math. Anal. Appl. \textbf{453}, 1005–1021 (2017)

\bibitem{mah} Maharana, S., Prajapat, J.K., Srivastava, H.M.: The radius of convexity of partial sums of convex functions in one direction, Proc. Natl. Acad. Sci. \textbf{87}, 215-219 (2017)

\bibitem{obra} Obradovi\'{c}, M., Ponnusamy, S., Wirths, K-J: Coefficient characterizations and sections for some univalent functions, Sib. Math. J. \textbf{54}, 679-696 (2013)



\end{thebibliography}

\end{document}